\documentclass[11pt, letterpaper]{amsart}
\usepackage[left=1in,right=1in,bottom=1in,top=1in]{geometry}
\usepackage{amsmath,amsthm,amssymb}
\usepackage{mathrsfs}
\usepackage[all,cmtip]{xy}
\usepackage{cite}
\usepackage{hyperref}
\usepackage{enumerate}
\usepackage{graphicx}

\newcommand{\mathfont}{\mathbf}
\newcommand{\ZZ}{\mathfont Z}

\newcommand{\FF}{\mathfont F}

\newcommand{\PP}{\mathfont{P}}

\DeclareFontFamily{OT1}{rsfs}{}
\DeclareFontShape{OT1}{rsfs}{n}{it}{<-> rsfs10}{}
\DeclareMathAlphabet{\mathscr}{OT1}{rsfs}{n}{it}

\theoremstyle{plain}
\newtheorem{lem}{Lemma}
\newtheorem{thm}[lem]{Theorem}
\newtheorem{prop}[lem]{Proposition}
\newtheorem{cor}[lem]{Corollary}

\theoremstyle{definition}

\newtheorem{example}[lem]{Example}
\newtheorem{remark}[lem]{Remark}
\newtheorem{notation}[lem]{Notation}

\numberwithin{equation}{section}
\numberwithin{lem}{section}

\DeclareMathOperator{\NP}{NP}

\DeclareMathOperator{\hodge}{Hodge}
\DeclareMathOperator{\GNP}{GNP}

\title[Realizing Artin-Schreier covers with minimal Newton polygons]
{Realizing Artin-Schreier covers of curves with minimal Newton polygons in positive characteristic}
\author{Jeremy Booher and Rachel Pries} 
\date{\today}

 \email{jeremy.booher@canterbury.ac.nz} 
 \address{School of Mathematics and Statistics\\
 University of Canterbury\\
   Christchurch, 8140, New Zealand}
   
   \email{pries@math.colostate.edu}
   \address{Department of Mathematics\\
   Colorado State University\\
   Fort Collins, Colorado 80523}

\usepackage[usenames,dvipsnames]{color}

\begin{document} 

\begin{abstract} 
Suppose $X$ is a smooth projective connected curve defined over an algebraically closed
field $k$ of characteristic $p>0$
and $B \subset X(k)$ is a finite, possibly empty, set of points.
The Newton polygon of a degree $p$ Galois cover of $X$ with branch locus $B$ 
depends on the ramification invariants of the cover.
When $X$ is ordinary, for every possible set of branch points and ramification invariants, 
we prove that there exists such a cover whose Newton polygon is minimal or close to minimal. 


MSC10: primary 11G20, 11M38, 14D15, 14H30, 14H40; secondary 11T23, 14D10, 14G17 
\end{abstract}

\maketitle

\section{Introduction}

Let $p$ be a prime and let $\FF$ be a finite field of cardinality $q=p^m$.
Suppose $X$ is a smooth projective geometrically connected curve of genus $g$ defined over $\FF$.
By the Weil conjectures, the zeta function of $X$ is a rational function of the form $Z(X, T) =L(X, T)/(1-T)(1-qT)$.
The $L$-polynomial $L(X, T)$ is a polynomial of degree $2g$ with integer coefficients.

The Newton polygon $\NP_X$ of $L(X,T)$ is the lower convex hull of the points $(i, v_p(c_i)/m)$, 
where $c_i$ is the coefficient of $x^i$ in $L(X, T)$, and $v_p(c_i)$ is the valuation of $c_i$ at $p$.
The Newton polygon $\NP_X$ is symmetric, with integral breakpoints, beginning at $(0,0)$ and ending at $(2g,g)$.
The Newton polygon is a geometric invariant; it can be defined more generally using the 
Dieudonn\'e module of the $p$-divisible group of ${\rm Jac}(X)$.  
In particular, it does not change under base extension of $\FF$.  
The slopes of the Newton polygon are the slopes of the constituent line segments.  
There is a natural partial ordering 
on the set of convex symmetric polygons with initial point $(0,0)$ and the same ending point;
we say that one polygon ``lies on or above'' another if the first lies geometrically above the second.  
The Newton polygon either stays the same or goes up under specialization.

In this paper, we study the Newton polygons of Artin-Schreier $\ZZ/p\ZZ$-covers $\pi : Y \to X$ that are branched at a finite set $B$ over an algebraically closed field $k=\bar{\FF}_p$.
Our goal is to prove, when $X$ is ordinary, that there exist Artin-Schreier $\ZZ/p\ZZ$-covers of $X$, 
with branch locus $B$
and having prescribed ramification above $B$, 
with ``low'' Newton polygon.   
(Recall that the curve $X$ is \emph{ordinary} if the number of $p$-torsion points in ${\rm Jac}(X)(k)$ is $p^g$.  Equivalently, the slopes of $\NP_X$ are the multi-set $\{0,1\}^g$ if and only if $X$ is ordinary.)

In \cite{KramerMiller}, Kramer-Miller gives a lower bound on the Newton polygon of an Artin-Schreier $\ZZ/p\ZZ$-cover 
$\pi :  Y \to X$ in terms of the Newton polygon of $X$ and the ramification of $\pi$.  
More precisely, let $B \subset X(k)$ be a finite, possibly empty, set of points.
Let $r=\#B$ and $\pi:Y \to X$ be a $\ZZ/p \ZZ$-cover with branch locus $B$.
For each $Q \in B$, let $d_Q$ be the \emph{ramification invariant} of $\pi$ above $Q$, namely the jump in the lower numbering
of the filtration of higher ramification groups above $Q$.  Note that $d_Q$ is a positive prime-to-$p$ integer.
Let $D=\{d_Q\}_{Q \in B}$.  In \cite[Corollary 1.2]{KramerMiller}, 
Kramer-Miller proves that the Newton polygon of $Y$ lies on or above 
the Newton polygon whose slopes are the multi-set
\begin{equation} \label{Elowerbound} 
\NP_X^{\hodge}(D):=\NP_X \cup \{0\}^{(p-1)(g+r-1)} \cup \{1\}^{(p-1)(g+r-1)} \cup 
\left(\bigcup_{Q \in B} \left \{ \frac{1}{d_Q}, \ldots, \frac{d_Q-1}{d_Q} \right \}^{p-1} \right).
\end{equation}
Here the exponent denotes the multiplicity of the slope.

When $X \simeq \PP^1$ and $B = \{ \infty \}$, results of Zhu, Blache, and F\'erard show that Kramer-Miller's lower bound for the Newton polygon is not optimal for many choices of ramification invariants $D = \{d\}$, \cite{Zhu03,BlacheFerard}.  
In particular, when $p \geq 3d$ the generic Newton polygon for the Artin-Schreier cover is known; denote it by $(p-1) \GNP(d,p)$ (see Notation~\ref{NBF} and Theorem~\ref{TBF}).  In many cases, $(p-1) \GNP(d,p)$ lies strictly above $\NP^{\hodge}_{\PP^1}(\{d\})$; see Example~\ref{ex:matching}. 

We consider a different Newton Polygon $\NP_X(D)$ whose slopes are the multi-set
\begin{equation} \label{Esmall}
\NP_X(D) := \NP_X \cup \{0\}^{(p-1)(g+r-1)} \cup \{1\}^{(p-1)(g+r-1)} \cup 
\left(\bigcup_{Q \in B} (p-1)\GNP(d_Q,p) \right).
\end{equation}
It follows from \cite[Theorem 7.4]{Robba} that $\NP_X^{\hodge}(D) \leq \NP_X(D)$. 
When $X \simeq \PP^1$, $B=\{\infty\}$, $D = \{d\}$, and $p \geq 3d$, the work of Zhu, Blache and F\'erard shows $\NP_X(D)=(p-1)\GNP(d,p)$ is the optimal lower bound for the Newton polygon of $Y$.

We use the theory of formal patching to prove the existence of Artin-Schreier $\ZZ/p\ZZ$-covers of an arbitrary ordinary curve, with an arbitrary branch locus and ramification invariants, 
that have ``low'' Newton polygon.

\begin{thm} \label{Tintro}
Suppose $X$ is a smooth projective connected curve defined over 
$k=\bar{\FF}_p$.  
Let $B \subset X(k)$ be a finite, possibly empty, set of points.
For $Q \in B$, let $d_Q$ be a positive prime-to-$p$ integer.

Suppose that $X$ is ordinary and that $p \geq {\rm max}\{3d_Q\}_{Q \in B}$.
Then there exists a $\ZZ/p\ZZ$-cover $\pi: Y \to X$, with branch locus $B$ and ramification invariants $D=\{d_Q\}_{Q\in B}$,
such that the Newton polygon $\NP_Y$ of $Y$ satisfies
\begin{equation}
\NP^{\hodge}_X(D) \leq \NP_Y \leq \NP_X(D).
\end{equation}
\end{thm}

Theorem~\ref{Tintro} follows directly from Theorem~\ref{Tmain} and \cite[Corollary 1.2]{KramerMiller}. 

\begin{remark}
\begin{enumerate}
\item  In the special case that $d_Q \mid (p-1)$ for every $Q \in B$, then
$\NP^{\hodge}_X(D) = \NP_X(D)$ and hence $\NP_Y$ equals $\NP^{\hodge}_X(D)$ 
which is minimal; see Corollary~\ref{cor:p1modd}.

\item With Kramer-Miller, we talked about whether the lower bound produced by his proof can
be strengthened to $\NP_X(D)$.
In situations where that is true, Theorem~\ref{Tmain} guarantees the existence of 
$\ZZ/p\ZZ$-covers with minimal Newton polygon.

\item In \cite{anumber}, we use the techniques of this paper to prove the existence of $\ZZ/p\ZZ$-covers 
with minimal $a$-number.

\end{enumerate}
\end{remark}

\begin{remark} 
The result of Kramer-Miller and the other results mentioned above are actually consequences of 
theorems about exponential sums.
There is a beautiful connection between Newton polygons of Artin-Schreier $\ZZ/p\ZZ$-covers of $X$ and Newton polygons of exponential sums of functions $f$ on $X$.
The theory of exponential sums is well-developed;
one major theme is to show that the Newton polygon ``lies on or above'' the Hodge polygon.
This was proven by Robba \cite[Theorem 7.4]{Robba} for ${\mathbb G}_m$ and by   
Wan \cite{Wan93}, especially Propositions 2.2 and 2.3, 
for many higher-dimensional tori. 
Under a congruence condition on $p$,
Sperber \cite[Theorem 3.11]{Sperber93} proved that the Newton polygon equals the Hodge polygon
in many situations; for example, this includes the case 
when $f$ is a polynomial of degree $d$ and $p \equiv 1 \bmod d$.
There are many important papers on this topic, 
including \cite{SZhu, Zhu03, Zhu04, Zhu04affinoid, BlacheFerard}. 

It is not clear if our proof contributes to the story about Newton polygons of exponential sums
associated with functions on curves over a fixed finite field.  
The issue is that it is not possible to control the field of definition of the cover
when using the technique of formal patching. \end{remark}

\subsection{Acknowledgments}
Booher was partially supported by the Marsden Fund Council administered by the Royal Society of New Zealand.
Pries was partially supported by NSF grant DMS-19-01819. 
We thank Bryden Cais, Joe Kramer-Miller, and Felipe Voloch for helpful conversations. 
We also thank the referee for the quick and helpful report.

\section{Initial statements}

As before, let $p$ be a prime, $\FF$ a finite field of cardinality $q = p^m$, and $X$ a smooth projective geometrically connected curve of genus $g$ defined over $\FF$. 
Without loss of generality, we suppose that the points in $B$ are defined over $\FF$.

\subsection{The $p$-rank}

The $p$-rank of $X$ is the integer $f_X$ such that $p^{f_X}$ is the number of $p$-torsion points 
in ${\rm Jac}(X)(\bar{\FF})$.  The curve $X$ is ordinary if $f_X = g$.
The $p$-rank equals the multiplicity of the slope $0$ in $\NP_X$.
More generally, the $p$-rank of a semi-abelian variety $A$ is 
$f_A={\rm dim}_{\FF_p} ({\rm Hom}(\mu_p, \bar{A}))$, 
where $\bar{A}$ is the base-change of $A$ to $k=\bar{\FF}$ and $\mu_p$ is the kernel of Frobenius on ${\mathbf G}_m$. 

\begin{lem} \label{Lprank}
Suppose $X$ is ordinary.  Let $\pi:Y \to X$ be a $\ZZ/p\ZZ$-cover, with branch locus $B$ and ramification invariants $D$.  Then the $p$-rank of $Y$ is 
the multiplicity of the slope $0$ in $\NP^{\hodge}_X(D)$ (and in $\NP_X(D)$).
\end{lem}

\begin{proof}
Let $r = \#B$.
By the Deuring-Shafarevich formula \cite[Theorem 4.2]{subrao}, \[f_Y-1 = p(f_X-1) + r(p-1).\]
Since $X$ is ordinary, $f_Y = 1 + p(g -1) + r(p-1)$.
So $f_Y = g + (p-1)(g-1+r)$, which is the multiplicity of the slope $0$ in $\NP^{\hodge}_X(D)$ and in $\NP_X(D)$.
\end{proof}

\subsection{An Artin-Schreier cover of a singular curve} \label{Sbuild}

Let $X$ and $B$ and $D$ be as above. 
In this section, we assume that $r \geq 2$ if $X \simeq \PP^1$. 
We build an Artin-Schreier cover
$\pi_\circ: Y_\circ \to X_\circ$ of singular curves, depending on the data of 
an unramified $\ZZ/p\ZZ$-cover $\pi': Y' \to X$ and for each $Q \in B$ a $\ZZ/p\ZZ$-cover $\pi_Q: Y_Q \to \PP^1$ branched only at $\infty$ where it has ramification invariant $d_Q$. 

By the Riemann-Hurwitz formula, if $Y'$ is connected, then its genus is 
\begin{equation} \label{Egenus}
g_{Y'} = g + (p-1)(g -1).
\end{equation}
Note that we allow $Y'$ to be disconnected; in fact, this is the only possibility when $f_X=0$.

Let $g_Q$ be the genus of $Y_Q$; by \cite[IV, Prop.\ 4]{serreLF}),
$g_Q = (p-1)(d_Q-1)/2$.

We build the singular curve $X_\circ$ by attaching a projective line to $X$ at each $Q \in B$ 
by identifying the point $0$ on $\PP^1$ and the point $Q$ in an ordinary double point.
Next, we build a $\ZZ/p \ZZ$-cover $\pi_\circ: Y_\circ \to X_\circ$.  
Let $Y_\circ$ be the singular curve whose components are $Y'$ and $Y_Q$ for $Q \in B$, formed 
by identifying the fiber of $\pi_Q$ above $0$ and the fiber of $\pi'$ above $Q$, in $p$ ordinary double points, 
in a Galois equivariant way.  
Let $\epsilon$ be the rank of the dual graph of $Y_\circ$, which is the minimal number of edges that need to be removed from the dual graph in order to form
a tree.

\begin{lem} 
The dual graph of $Y_\circ$ is a bipartite graph with $r$ vertices on the left side.
\begin{enumerate}
\item If $Y'$ is connected, then there is one vertex on the right side and it is connected to each of the vertices on the left
with $p$ edges.  Also $\epsilon= r(p-1)$.
\item If $Y'$ is disconnected, then there are $p$ vertices on the right side, each of which is connected to 
each vertex on the left with a unique edge.  Also $\epsilon=(r-1)(p-1)$.
\end{enumerate}
\end{lem}

\begin{proof}
The facts about the vertices and edges are immediate from the construction of $Y_\circ$.
To compute $\epsilon$, we determine the minimal number of edges that need to be removed from the dual graph in order to form a tree.
When $Y'$ is connected, then we need to remove $p-1$ edges
between each vertex on the left and the vertex on the right.
When $Y'$ is disconnected, we pick a distinguished vertex on the left and the right; 
then we need to remove all edges between the other $r-1$ vertices on the left and the other $p-1$ vertices on the right.
\end{proof}

The next result is immediate from \cite[9.2.8]{BLR}.

\begin{prop} \label{Pextension}
With notation as above:
\begin{enumerate}
\item If $Y'$ is connected, then ${\rm Jac}(Y_\circ)$ is an extension of 
$\displaystyle {\rm Jac}(Y') \oplus \left (\bigoplus_{Q \in B} {\rm Jac}(Y_Q) \right) $ by a torus of rank $\epsilon=r(p-1)$.  
\item If $Y'$ is disconnected, then ${\rm Jac}(Y_\circ)$ is an extension of 
$\displaystyle {\rm Jac}(X)^p \oplus \left (\bigoplus_{Q \in B} {\rm Jac}(Y_Q) \right)$ by a torus of rank $\epsilon=(r-1)(p-1)$. 
\end{enumerate}
In both cases, ${\rm Jac}(Y_\circ)$ 
is a semi-abelian variety of dimension 
\[g_{Y_\circ}=g+(p-1)(g+r-1) + \sum_{Q \in B}g_{Q}.\] 
\end{prop}

\subsection{Formal patching} \label{Sfp}

In Section~\ref{Sbuild}, we constructed a $\ZZ/p\ZZ$-cover $\pi_\circ:Y_\circ \to X_\circ$ of singular curves.
In this section, we consider $\pi_\circ$ to be defined over $k = \bar{\FF}$.  

\begin{prop} \label{Pharbater}
The $\ZZ/p\ZZ$-cover $\pi_\circ: Y_\circ \to X_\circ$ of singular curves has a flat deformation to
a $\ZZ/p \ZZ$-cover $\pi'':Y'' \to X$ defined over $k$ such that $Y''$ is smooth and connected,  
$\pi''$ has branch locus $B$, and $\pi''$ has ramification invariant $d_Q$ above each $Q \in B$.
\end{prop}

\begin{proof}
This is immediate from the theory of formal patching, see \cite{harbaterstevenson}.
\end{proof}

\begin{remark}
For the convenience of the reader who is not familiar with formal patching, we include some details of the proof.

Let $R=k[[t]]$, which has fraction field $K=k((t))$ and residue field $k$. 
Let $S={\rm Spec}(R)$, with generic point $S_K = {\rm Spec}(K)$ and closed point $S_\circ = {\rm Spec}(k)$.
Consider the $S$-curve $X_S=X \times_{S_\circ} S$
and the sections $Q_S=Q \times_{S_\circ} S$.  Let $B_S = B \times_{S_\circ} S = \cup_{Q \in B} Q_S$.
We first define an $S$-curve ${\mathcal X}_S$ whose generic fiber 
is $X_S \times_S S_K$ and whose special fiber is $X_\circ$. 
The curve ${\mathcal X}_S$ is called a \emph{thickening} of $X_\circ$.

To do this, for each $Q \in B$, we 
blow-up the point $Q \in B_\circ$ on the special fiber $X_{S_\circ}$ to a projective line, which we denote by $X_Q$. 
Without loss of generality, we can assume that:\\
(i) the closure of $Q_K$ intersects $X_Q$ at the point at infinity, denoted by $\infty_Q$, on $X_Q$; and\\
(ii) the strict transform $X_{str}$ of $X_{S_\circ}$ intersects $X_Q$ 
at the point $0$, denoted by $0_Q$, on $X_Q$.\\
Let $\mathcal{X}_S$ denote this blowup.

Next, we consider the complete local rings of ${\mathcal X}_S$ at some points of interest.
Let $R_{Q,0}$ (resp.\ $R_{Q, \infty}$) be the complete local ring of ${\mathcal X}_S$ at $0_Q$ (resp.\ $\infty_Q$).
Let $x_Q$ be a parameter on $X_Q$ that has a zero of order $1$ at $0_Q$ and has a pole of order $1$ at $\infty_Q$.
Let $x_Q'$ be a parameter on $X_{str}$ that has a zero of order $1$ at $Q$.
Without loss of generality, we can assume that:\\
(iii) $R_{Q,0} \simeq k[[x_Q,x_Q',t]]/\langle x_Qx_Q'-t\rangle$; and\\
(iv) $R_{Q, \infty} \simeq k[[x_Q^{-1}, t]]$, and this isomorphism identifies the ideal $\langle x_Q^{-1} \rangle$
with the point $Q_S$\footnote{or, more precisely, 
the closure of $Q_K$ in ${\mathcal X}_S$, base changed to ${\rm Spec}(R_{Q,\infty})$.}.

Consider the $\ZZ/p\ZZ$-cover $\pi_\circ:Y_\circ \to X_\circ$ from Section~\ref{Sbuild}.
Its branch locus is $\{\infty_Q \mid Q \in B\}$ and $X_\circ$ is singular at the points $\{0_Q \mid Q \in B\}$.
The important step is to thicken the cover $\pi_\circ$ at these points.
To thicken $\pi_\circ$ at $0_Q$ for $Q \in B$, let $A_{Q,0}$ be the ring such that
${\rm Spec}(A_{Q,0}) = {\rm Ind}_{\{0\}}^{\ZZ/p\ZZ} {\rm Spec}(R_{Q,0})$; the natural map 
$\psi_{Q, 0}: {\rm Spec}(A_{Q,0}) \to {\rm Spec}(R_{Q,0})$ is an \'etale $\ZZ/p\ZZ$-cover.

To thicken $\pi_\circ$ at $\infty_Q$ for $Q \in B$, 
consider the restriction of $\pi_Q:Y_Q \to X_Q$ to ${\rm Spec}(k[[x_Q^{-1}]])$.
The corresponding extension of function fields is given by
an equation of the form $y^p_Q-y_Q=f_Q$ for some 
$f_Q \in {\rm Frac}(k[[x_Q^{-1}]])= k((x_Q^{-1}))$.
By Artin-Schreier theory, since every element of $k[[x_Q^{-1}]]$ is of the form $\alpha^p- \alpha$, 
without loss of generality, we can suppose that $f_Q \in k[x_Q]$.
Furthermore, by Artin-Schreier theory, we can suppose $f_Q$ has degree $d_Q$.
We view $f_Q$ as a function on ${\rm Spec}(R_{Q, \infty})$ which has a pole only on the divisor 
$x_Q^{-1}=0$.  The equation $y^p_Q-y_Q-f(x_Q)$ then defines a ring $A_{Q, \infty}$ and 
a $\ZZ/p\ZZ$-cover $\psi_{Q, \infty}: {\rm Spec}(A_{Q, \infty}) \to {\rm Spec}(R_{Q, \infty})$, 
which is branched only over $x_Q^{-1}=0$; the cover $\psi_{Q, \infty}$
can be viewed as a constant deformation of the restriction of $\pi_Q$ near $\infty_Q$.

The data of $\pi_\circ$, together with the data of $\psi_{Q, 0}$ and $\psi_{Q, \infty}$ for $Q \in B$, 
defines a \emph{thickening problem} as in \cite[page 288]{harbaterstevenson}.
By \cite[Theorem 4]{harbaterstevenson}, there exists a solution to the thickening problem.
This solution is a $\ZZ/p\ZZ$-cover $\tilde{\pi}:{\mathcal Y} \to {\mathcal X}$ of $S$-curves, 
whose special fiber is $\pi_\circ$. 
Furthermore, by \cite[Theorem 5]{harbaterstevenson}, 
over the generic fiber, the branch locus of $\tilde{\pi}_{S_K}$ is 
$B_{S_K} = B_S \times_S S_K$, 
and the ramification invariant above the generic geometric point of $Q_S$ is $d_Q$.

By Artin's algebraization theorem, the cover $\tilde{\pi}$ descends to ${\rm Spec}(L)$ where 
$L$ is a finite extension of $k[t]$.  Choose a closed point $s''$ of ${\rm Spec}(L)$.
Then the fiber $\pi''$ of $\tilde{\pi}_S$ over $s''$ is a cover satisfying the properties in Proposition~\ref{Pharbater}.
\end{remark}

\section{Covers of the affine line}

Consider the special case that $X \simeq \PP^1$ and $B=\{\infty\}$ and $d \geq 1$ is a prime-to-$p$ integer.  
In this case, one considers exponential sums associated with a polynomial $f$ in one variable and having degree $d$.  

For any non-trivial additive character $\psi$ of $\FF$ and any natural number $\ell$, 
one can define the exponential sum $S_\ell(f, \psi)$ as follows.
Let $\FF_\ell$ denote the unique field which is a degree $\ell$ extension of $\FF$; it has cardinality $q^\ell$.
Let $\psi_\ell = \psi \circ {\rm Tr}_{\FF_\ell/\FF}$, which is a non-trivial additive character of $\FF_\ell$. 
We consider the exponential sums $S_\ell(f, \psi):= \sum_{a \in \FF_\ell}\psi_\ell(f(a))$ and the $L$-polynomial 
$L(f, \psi, T) := {\rm exp}(\sum_{\ell=1}^\infty S_\ell(f, \psi) T^\ell/\ell)$.
This $L$-polynomial has a Newton polygon, denoted $\NP_{f, \psi}$.

Zhu introduced the following explicit Newton polygons \cite[Section 4, page 679]{Zhu03}.

\begin{notation} \label{NBF} For $1 \leq n \leq d-1$ let
\begin{equation} \label{EYn}
Y_n := \min_{\sigma \in S_n} \sum_{k=1}^n \left \lceil \frac{pk-\sigma(k)}{d} \right\rceil.
\end{equation}
Let $\GNP(d,p)$ be the lower convex hull of $(0,0)$ and $(n, \frac{Y_n}{p-1})$ for $1 \leq n \leq d-1$.\\
Let $(p-1) \GNP(d,p)$ be the lower convex hull of $(0,0)$ and $((p-1)n, Y_n)$ for $1 \leq n \leq d-1$.
\end{notation}

Zhu proved that $\GNP(d,p)$ is the Newton polygon for the exponential sum associated to a generic $f$ provided $p$ is sufficiently large \cite[Theorem 5.1]{Zhu03}.
This was extended by Blache and F\'erard, who proved that 
$\GNP(d,p)$ occurs for all $f$ in an explicitly defined Zariski open subset,
when $p \geq 3d$ \cite{BlacheFerard}.  We state a corollary of their results for covers.

\begin{thm} \label{TBF} 
\cite[Theorem 4.1]{BlacheFerard}
Suppose $\FF$ is a finite field of characteristic $p$ with $p \geq 3d$.
Then there exists a $\ZZ/p\ZZ$-cover $\pi:Y \to \PP^1$ defined over $\FF$, which is branched only above $\infty$ and where 
it has ramification invariant $d$, such that the Newton polygon of $Y$ 
equals the generic Newton polygon
$(p-1) \GNP(d,p)$.
\end{thm}

\begin{proof}
For $f \in \FF[x]$ of degree $d$, the Artin-Schreier equation $y^p-y=f$ defines 
a $\ZZ/p \ZZ$-cover $\pi:Y \to \PP^1$ branched only at $\infty$ with ramification invariant $d$.
Without loss of generality, suppose $f$ is monic with no constant term and the coefficient of $x^{d-1}$ is $0$.

The Newton polygon of $Y$ is $(p-1)\NP_{f, \psi}$, where $p-1$ is a scaling factor on the Newton polygon 
or, equivalently, on the multiplicities of the slopes of the Newton polygon.\footnote{This is because the zeta function of $Y$ factors as 
$Z(Y, T) = Z(X,T) \prod_{\psi} L(f, \psi, T)$,
where $\psi$ ranges over the non-trivial characters $\ZZ/p \ZZ \to \ZZ_p[\zeta_p]^\times$. }

By \cite[Theorem 4.1]{BlacheFerard}, 
if the coefficients of $f$ are in an (explicitly determined) open dense subset of ${\mathbb A}^{d-2}$,
then $\NP_{f, \psi}={\rm GNP}(d,p)$.
Thus $\NP_Y = (p-1){\rm GNP}(d,p)$.
\end{proof}

The lower bound $\NP^{\hodge}_{\PP^1}(\{d\})$ defined in \eqref{Elowerbound} does not equal $(p-1)\GNP(d,p)$ in general.
One reason for this is that the Newton polygon of a curve must have integer breakpoints.  
However, this is not the only reason; in some cases there are symmetric Newton polygons  
starting at $(0,0)$ and ending at $(2g,g)$ with integer breakpoints which lie strictly between 
$\NP^{\hodge}_{\PP^1}(\{d\})$ and $(p-1)\GNP(d,p)$ in the natural partial ordering.
We thank Joe Kramer-Miller for the following example.

\begin{example} \label{ex:matching}
Let $p=23$ and $d=6$.  Then $\NP^{\hodge}_{\PP^1}(\{d\})$ is the Newton polygon with vertices 
$$(0, 0), (22, 11/3), (44, 11), (66, 22), (88, 110/3), \text{ and } (110, 55).$$
Furthermore, $(p-1) \GNP(d,p)$ is the Newton polygon with vertices 
$$(0, 0), (22, 4), (44, 12), (66, 23), (88, 37),  \text{ and } (110, 55).$$
But the Newton polygon with vertices
$$(0, 0), (22, 4), (44, 11), (66, 22), (88, 37),\text{ and } (110, 55)$$
is symmetric, has integer breakpoints, and lies between them.
\end{example}

\begin{remark} \label{rmk:p1modd}
This is not an issue when $p \equiv 1 \pmod{d}$.  In that situation, it is known that $\NP^{\hodge}_{\PP^1}(\{d\})$ equals $(p-1)\GNP(d,p)$; see for example \cite[Remark 4.1]{BlacheFerard}.
\end{remark}

\section{The main result}

Recall the definition of $\NP_X(D)$ from \eqref{Esmall}.

\begin{thm} \label{Tmain}
Suppose $X$ is a smooth projective connected curve defined over 
$k=\bar{\FF}_p$.  
Let $B \subset X(k)$ be a finite, possibly empty, set of points.
For $Q \in B$, let $d_Q$ be a positive prime-to-$p$ integer.

Suppose $L_{\NP}(D)$ is a fixed lower bound for the Newton polygon of a 
$\ZZ/p\ZZ$-cover $\pi: Y \to X$ with branch locus $B$ and ramification invariants $\{d_Q\}_{Q \in B}$.
If $X$ is ordinary and 
$p \geq {\rm max}\{3d_Q\}_{Q \in B}$, 
then there exists a $\ZZ/p\ZZ$-cover $\pi: Y \to X$ of smooth curves over $k$, 
with branch locus $B$ and ramification invariants $D=\{d_Q\}_{Q\in B}$,
such that the Newton polygon of $Y$ satisfies
\begin{equation}
L_{\NP}(D) \leq \NP_Y \leq \NP_X(D).
\end{equation}
\end{thm}

\begin{remark}
We need the lower bound on $p$ in order to apply the results of \cite{BlacheFerard};
it is possible that this condition can be removed.
\end{remark}

\begin{proof}
For each $Q \in B$, if $p \geq 3 d_Q$, then by Theorem \ref{TBF}
there exists a $\ZZ/p \ZZ$-cover $\pi_Q: Y_Q \to {\mathbb P}^1$, 
branched only at $\infty$, with 
ramification invariant $d_Q$ and Newton polygon $(p-1)\GNP(d_Q,p)$. 

First suppose $g>0$.
Since $X$ is ordinary, there is an unramified $\ZZ/p \ZZ$-cover $\pi': Y' \to X$ where $Y'$ is connected.
The genus of $Y'$ is $g_{Y'} = g + (p-1)(g -1)$ by \eqref{Egenus}.
By the Deuring-Shafarevich formula, if $X$ is ordinary then $Y'$ is ordinary, so its Newton polygon
has $g_{Y'}$ slopes of $0$ and of $1$.

Construct the cover $\pi_\circ: Y_\circ \to X_\circ$ as in Section \ref{Sbuild}, using 
the inputs of the covers $\pi':Y' \to X$ and $\pi_Q:Y_Q \to \PP^1$ for $Q \in B$.  
By Proposition \ref{Pextension}(1), ${\rm Jac}(Y_\circ)$ is an extension of 
${\rm Jac}(Y') \oplus \left(\bigoplus_{Q \in B} {\rm Jac}(Y_Q)\right) $ by a torus of rank $\epsilon=r(p-1)$.  
The torus increases the $p$-rank of ${\rm Jac}(Y_\circ)$ by $\epsilon$.  
So the slopes of the Newton polygon of $Y_\circ$ are 
\[
\{0\}^{g_{Y'} + r(p-1)} \cup \{1\}^{g_{Y'} + r(p-1)} \cup \bigcup_{Q \in B} (p-1)\GNP(d_Q,p).
\]

If $g=0$, we instead take $\pi': Y' \to X$ be a disconnected $\ZZ/p\ZZ$-cover.
By Proposition \ref{Pextension}(2), ${\rm Jac}(Y_\circ)$ is an extension of 
${\rm Jac}(X)^p \oplus \left(\bigoplus_{Q \in B} {\rm Jac}(Y_Q)\right)$ by a torus of rank $\epsilon=(r-1)(p-1)$. 
In this case, the slopes of the Newton polygon of $Y_\circ$ are
\[
\{0\}^{pg + (r-1)(p-1)} \cup \{ 1\}^{pg + (r-1)(p-1)} \cup \bigcup_{Q \in B} (p-1)\GNP(d_Q,p).
\]

In either case, 
the Newton polygon of $Y_\circ$ equals $\NP_X(D)$. 
By Proposition \ref{Pharbater}, 
the $\ZZ/p \ZZ$-cover $\pi_\circ:Y_\circ \to X_\circ$ has a flat deformation to a 
$\ZZ/p \ZZ$-cover $\pi'':Y'' \to X$ defined over $k$ 
such that $Y''$ is smooth and connected,  
$\pi''$ has branch locus $B$, and $\pi''$ has ramification invariant $d_Q$ above each $Q \in B$.
The Newton polygon can only go down in such a deformation by a result of Grothendieck and Katz, 
see \cite[Theorem 2.3.1]{katzslope}.
By hypothesis, $L_{\NP}(D)$ is a lower bound for the 
Newton polygon of such a cover. So the Newton polygon of $Y''$ lies between $L_{\NP}(D)$ and $\NP_X(D)$.
\end{proof}

Theorem~\ref{Tintro} follows immediately from Theorem~\ref{Tmain}, taking $L_{\NP}(D) = \NP_X^{\hodge}(D)$ and using Kramer-Miller's result \cite[Corollary 1.2]{KramerMiller}.

\begin{cor} \label{cor:p1modd}
With the notation and hypotheses of Theorem~\ref{Tintro}, suppose furthermore that $p \equiv 1 \pmod{d_Q}$ for every $Q \in B$.  Then there exists a $\ZZ/p \ZZ$-cover $\pi : Y \to X$ with branch locus $B$ and ramification invariants $D = \{d_Q\}_{Q \in B}$, 
such that $Y$ has Newton polygon $\NP^{\hodge}_X(D) = \NP_X(D)$.
\end{cor}

In other words, Corollary \ref{cor:p1modd} shows that the lower bound $\NP^{\hodge}_X(D)$ is sharp in this special case.

\begin{proof}
By Remark~\ref{rmk:p1modd}, $\NP_X^{\hodge}(D) = \NP_X(D)$ under this congruence condition on $p$.
Thus the result follows from Theorem~\ref{Tintro}.
\end{proof}

\begin{remark}
Suppose $g > 1$ and $X$ is not ordinary.  If $\pi:Y \to X$ is a $\ZZ/p\ZZ$-cover, it is not 
currently known which Newton polygons can occur for $Y$, even if the cover is unramified.  
As noted in \cite[\S1.4]{KramerMiller}, the lower bound $\NP^{\hodge}_X(D)$ is not the correct lower bound 
for the Newton polygon of $Y$ as it has too many segments of slope $0$ compared with Lemma~\ref{Lprank}.
\end{remark}

\begin{remark}
In the proof of Theorem \ref{Tmain}, it is also possible to use a disconnected cover $\pi':Y' \to X$ when $g >0$. 
We chose a connected cover since it adds an additional level of control that could be useful in future applications.
\end{remark}


\begin{thebibliography}{AMBB{\etalchar{+}}}

\bibitem[AMBB{\etalchar{+}}]{anumber}
Fiona Abney-McPeek, Hugo Berg, Jeremy Booher, Sun~Mee Choi, Viktor Fukala,
  Miroslav Marinov, Theo Müller, Paweł Narkiewicz, Rachel Pries, Nancy Xu,
  and Andrew Yuan, \emph{Realizing {A}rtin-{S}chreier covers with minimal
  a-numbers in characteristic p}, preprint.

\bibitem[BF07]{BlacheFerard}
R\'{e}gis Blache and \'{E}ric F\'{e}rard, \emph{Newton stratification for
  polynomials: the open stratum}, J. Number Theory \textbf{123} (2007), no.~2,
  456--472. \MR{2301225}

\bibitem[BLR90]{BLR}
Siegfried Bosch, Werner L{\"u}tkebohmert, and Michel Raynaud, \emph{N\'eron
  models}, Ergebnisse der Mathematik und ihrer Grenzgebiete (3) [Results in
  Mathematics and Related Areas (3)], vol.~21, Springer-Verlag, Berlin, 1990.
  \MR{1045822 (91i:14034)}

\bibitem[HS99]{harbaterstevenson}
David Harbater and Katherine~F. Stevenson, \emph{Patching and thickening
  problems}, J. Algebra \textbf{212} (1999), no.~1, 272--304. \MR{1670658}

\bibitem[Kat79]{katzslope}
Nicholas~M. Katz, \emph{Slope filtration of {$F$}-crystals}, Journ\'ees de
  {G}\'eom\'etrie {A}lg\'ebrique de {R}ennes ({R}ennes, 1978), {V}ol. {I},
  Ast\'erisque, vol.~63, Soc. Math. France, Paris, 1979, pp.~113--163.
  \MR{563463}

\bibitem[KM]{KramerMiller}
Joe Kramer-Miller, \emph{$p$-adic estimates of exponential sums on curves},
  arXiv:1909.06905.

\bibitem[Rob84]{Robba}
Philippe Robba, \emph{Index of {$p$}-adic differential operators. {III}.
  {A}pplication to twisted exponential sums}, no. 119-120, 1984, $p$-adic
  cohomology, pp.~7, 191--266. \MR{773094}

\bibitem[Ser79]{serreLF}
Jean-Pierre Serre, \emph{Local fields}, Graduate Texts in Mathematics, vol.~67,
  Springer-Verlag, New York-Berlin, 1979, Translated from the French by Marvin
  Jay Greenberg. \MR{554237}

\bibitem[Spe86]{Sperber93}
Steven Sperber, \emph{On the {$p$}-adic theory of exponential sums}, Amer. J.
  Math. \textbf{108} (1986), no.~2, 255--296. \MR{833359}

\bibitem[Sub75]{subrao}
Dor{\'e} Subrao, \emph{The {$p$}-rank of {A}rtin-{S}chreier curves},
  Manuscripta Math. \textbf{16} (1975), no.~2, 169--193. \MR{0376693}

\bibitem[SZ02]{SZhu}
Jasper Scholten and Hui~June Zhu, \emph{The first slope case of {W}an's
  conjecture}, Finite Fields Appl. \textbf{8} (2002), no.~4, 414--419.
  \MR{1933613}

\bibitem[Wan93]{Wan93}
Da~Qing Wan, \emph{Newton polygons of zeta functions and {$L$}-functions}, Ann.
  of Math. (2) \textbf{137} (1993), no.~2, 249--293. \MR{1207208}

\bibitem[Zhu03]{Zhu03}
Hui~June Zhu, \emph{{$p$}-adic variation of {$L$}-functions of one variable
  exponential sums. {I}}, Amer. J. Math. \textbf{125} (2003), no.~3, 669--690.
  \MR{1981038}

\bibitem[Zhu04a]{Zhu04}
\bysame, \emph{Asymptotic variation of {$L$}-functions of one-variable
  exponential sums}, J. Reine Angew. Math. \textbf{572} (2004), 219--233.
  \MR{2076126}

\bibitem[Zhu04b]{Zhu04affinoid}
\bysame, \emph{{$L$}-functions of exponential sums over one-dimensional
  affinoids: {N}ewton over {H}odge}, Int. Math. Res. Not. (2004), no.~30,
  1529--1550. \MR{2049830}

\end{thebibliography}

\newcommand{\etalchar}[1]{$^{#1}$}
\providecommand{\bysame}{\leavevmode\hbox to3em{\hrulefill}\thinspace}
\providecommand{\MR}{\relax\ifhmode\unskip\space\fi MR }
\providecommand{\MRhref}[2]{%
  \href{http://www.ams.org/mathscinet-getitem?mr=#1}{#2}
}
\providecommand{\href}[2]{#2}

\end{document}